\documentclass[11pt,a4paper,reqno]{amsart}

\usepackage{amsmath}
\usepackage{amsfonts}
\usepackage{palatino}
\usepackage{a4wide}
\usepackage{amssymb}
\usepackage{amsthm}

\usepackage[]{draftcopy}
\draftcopyName{ Draft}{240}
\draftcopySetScale{270}

\theoremstyle{plain}

\newtheorem{teo}{Theorem}[section]
\newtheorem{lemma}[teo]{Lemma}

\newtheorem{ackn}{Acknowledgments\!}

\theoremstyle{definition}

\theoremstyle{remark}

\newtheorem{rem}[teo]{Remark}

\numberwithin{equation}{section}

\def\SS{{{\mathbb S}}}

\def\RR{{\mathbb R}}

\def\RRR{{\mathrm R}}

\def\HHH{{\mathrm H}}
\def\TTT{{\mathrm{T}}}
\def\Ric{{\mathrm {Ric}}}

\def\Rm{{\mathrm {Rm}}}

\def\trace{{\mathrm{tr}}}

\title[A Note on Codazzi Tensors]{A Note on Codazzi Tensors}

\author[Giovanni Catino]{Giovanni Catino}
\address[Giovanni Catino]{Dipartimento di Matematica, Politecnico di Milano, Piazza Leonardo da Vinci 32, Milano, Italy, 20133}
\email[G. Catino]{giovanni.catino@polimi.it}

\author[Carlo Mantegazza]{Carlo Mantegazza}
\address[Carlo Mantegazza]{Scuola Normale Superiore di Pisa, Piazza dei Cavalieri 7, Pisa, Italy, 56126}
\email[C. Mantegazza]{c.mantegazza@sns.it}

\author[Lorenzo Mazzieri]{Lorenzo Mazzieri}
\address[Lorenzo Mazzieri]{Scuola Normale Superiore di Pisa, Piazza dei Cavalieri 7, Pisa, Italy, 56126}
\email[L. Mazzieri]{l.mazzieri@sns.it}

\date{\today}

\begin{document}

\begin{abstract} We discuss a gap in Besse's book~\cite{besse},
  recently pointed out by Merton in~\cite{merton}, which concerns the
  classification of Riemannian manifolds admitting a Codazzi tensors
  with exactly two distinct eigenvalues. For such manifolds, we prove
  a structure theorem, without adding extra hypotheses and then we
  conclude with some application of this theory to the classification
  of three--dimensional gradient Ricci solitons.
\end{abstract}

\maketitle
\tableofcontents

\section{Introduction}

For $n\geq 3$, let $(M^{n},g)$ be a smooth Riemannian manifold and
consider a Codazzi tensor $\TTT$ on $M^{n}$, i.~e., a symmetric
bilinear form satisfying the Codazzi equation 
$$
(\nabla_{X}\TTT)(Y,Z)=(\nabla_{Y}\TTT)(X,Z)\, ,
$$ 
for every tangent vectors $X,Y,Z$.

In the book {\em Einstein Manifolds}~\cite{besse}, by Besse, it is proved that if a Riemannian manifold $(M^n,g)$ admits a Codazzi tensor $\TTT$ such that at every point of $M^n$, $\TTT$ has
exactly two distinct eigenvalues, then
\begin{itemize}
\item if the constant multiplicities of the two eigenspaces are larger than one, $(M^n,g)$ is locally a Riemannian product,
\item if the above multiplicities are respectively 1 and $n-1$ {\em and the trace of $\TTT$ is constant}, then 
$(M^n,g)$ is locally a warped product of an $(n-1)$--dimensional Riemannian manifolds on an interval of $\RR$.
\end{itemize}
For more details, we refer the reader to discussion~16.12 in~\cite{besse}.

Before showing this result, Besse states "{\em ... a similar argument
  works without this hypothesis [that trace of $\TTT$ is 
constant]}". Recently in~\cite{merton}, G.~Merton provided a
counterexample to the local warping structure, showing that the last
Besse's statement is false. In~\cite{merton}, he also discusses some
possible extra hypotheses, weaker than {\em trace of $\TTT$ constant},
under which the local warped structure can be obtained.

Our goal here is to describe, without adding extra hypotheses to
Besse's statement, what is the local geometric structure of a
Riemannian manifold admitting a Codazzi tensor with exactly two
distinct eigenvalues. Essentially, one has that the manifold may
present zones where it is a warped product on a interval and zones
where it is not. In this latter case, it turns out that the manifold
admits a local totally geodesics foliation. This is the content of
our Theorem~\ref{mainteo}.

In Section~\ref{example}, we will give an example of a Riemannian
manifold where both the situations (local warped product 
structure and local totally
geodesics foliation) described in our structure theorem are present at
the same time.

Finally, in the last section, we will show how this Codazzi tensors
theory can be applied to the classification of gradient Ricci
solitons.

\medskip

\begin{ackn} 
The authors are members of the Gruppo Nazionale per
l'Analisi Matematica, la Probabilit\`{a} e le loro Applicazioni (GNAMPA) of the Istituto Nazionale di Alta Matematica (INdAM). They are supported by the GNAMPA project ``Equazioni di evoluzione geometriche e strutture di tipo Einstein''.
\end{ackn}

\bigskip

\section{Codazzi Tensors with Two Distinct Eigenvalues}

In this section we present the statement and the proof of our main theorem.

\begin{teo}
\label{mainteo} 
Let $\TTT$ be a Codazzi tensor on $(M^{n},g)$, with $n\geq 3$. Suppose
that at every point of $M^n$, the tensor $\TTT$ has exactly two
distinct eigenvalues $\rho$ and $\sigma$ of multiplicity 1 and $n-1$,
respectively. Finally, we let $W=\{ p\in M^n \, \big| \, d\sigma(p)
\neq 0 \}$. Then, we have that 
\begin{enumerate}
\item The closed set $\overline{W}=W\cup \partial W$ with the metric
  $g|_{\overline{W}}$ is locally isometric to the warped product of
  some $(n-1)$--dimensional Riemannian manifold on an interval of
  $\RR$ and $\sigma$ is constant along the "leaves" of the warped
  product.
\item The boundary of $W$, if present, 
is given by the disjoint union of connected totally geodesic hypersurfaces where $\sigma$ is constant. 
\item Each connected component of the complement of $\overline{W}$ in
  $M$, if present, has $\sigma$ constant and it is foliated by totally
  geodesic hypersurfaces.
\end{enumerate}
The $(n-1)$--dimensional tangent subspaces to the above warping hypersurfaces at point (1) and to the totally geodesic hypersurfaces at points (2) and (3) are the eigenspaces of $\TTT$ with respect to $\sigma$.
\end{teo}
\begin{proof} Since the Codazzi tensor $\TTT$ has
exactly two distinct eigenvalues $\rho$ and $\sigma$ of multiplicity 1 and $n-1$, respectively, we have by Proposition~16.11 in~\cite{besse} that the tangent bundle $TM$ of $M$ splits as the orthogonal direct sum of two
{\em integrable} eigendistributions: a line field 
$V_{\rho}$ and a codimension one distribution $V_{\sigma}$ with totally {\em umbilical} leaves, which means 
that the second fundamental form ${h}$ of each leaf is proportional to
the metric $g^{\sigma}$, induced by $g$ on $V_{\sigma}$.

To fix the notations, we will denote by ${\nabla}$ the Levi--Civita connection of the metric $g$ on $M^n$ and we recall that the (scalar) second fundamental form of a leaf $L$ of the codimension one distribution $V_{\sigma}$ can be defined as
$$
h(X,Y)=-g(\nabla_X Y,\nu) \, ,
$$
where $X$ and $Y$ are vector fields along $L$ and $\nu$ is a choice of a unit normal vector field to $L$. The fact that $L$ is umbilical means that, for every couple of vector fields $X,Y$ tangent to $L$, we have
$$
h(X,Y) \, = \, \frac{\HHH}{n-1} \, g^{\sigma}(X,Y) \, ,
$$
where $\HHH$, the mean curvature of $L$, is defined as the trace of $h$ with respect to $g^{\sigma}$.

Since $n\geq 3$, we have that the codimension one distribution
$V_\sigma$ has dimension strictly bigger than one. Thus, we infer from
Proposition~16.11 in~\cite{besse} that the eigenfunction $\sigma$ must
be constant along the leaves of $V_\sigma$. In particular, whenever
$d\sigma \neq 0$, the leaves of $V_\sigma$ are locally 
regular level sets of $\sigma$.

To proceed, we fix a point $p \in M$ and we consider a local
coordinate system $(x^0,\ldots, x^{n-1})$ adapted to the leaves of
$V_\sigma$ on a neighborhood $U$ of $p$. This means that
$\partial/\partial x^0 \in V_\rho$ and $\partial/\partial x^j \in
V_\sigma$, for $j= 1,\ldots, n-1$. In this chart, the unit vector
field $\nu = (\partial / \partial x^0) / \sqrt{g_{00}}$ is normal to
any leaf of the distribution $V_\sigma$ and since the two
eigendistributions are mutually orthogonal we immediately get $g_{0j}
= 0$ and $\TTT_{0j}=0$, for $j=1, \ldots,n-1$. If $L$ is the leaf of
$V_\sigma$ through the point $p$, the second fundamental form of $L$
about $p$ and the umbilicity condition can be written as
\begin{equation}
\label{umbilic}
{h}_{ij} \, = \,-\big\langle {\nabla}_{\frac{\partial\,}{\partial x^{i}}} \tfrac{\partial\,}{\partial x^{j}}, \nu\big\rangle \, = \,  
 \,-\big\langle {\nabla}_{\frac{\partial\,}{\partial x^{i}}} \tfrac{\partial\,}{\partial x^{j}}, \tfrac{\partial\,}{\partial x^{0}} \big\rangle/\sqrt{g_{00}} \, = \,  -{\Gamma}^{0}_{ij}\sqrt{g_{00}}= \frac{\HHH}{n-1} \, g_{ij}^\sigma\,,
\end{equation}
for $i,j = 1,\ldots,n-1$.\\
Denoting by ${\nabla}^{\sigma}$ the Levi--Civita connection of the induced metric $g^{\sigma}$, the Codazzi--Mainardi equations (see Theorem~1.72 in~\cite{besse}) read
\begin{equation}\label{codman}
\big({\nabla}^{\sigma}_{\frac{\partial\,}{\partial x^{i}}}{h}\big) \big(\tfrac{\partial\,}{\partial x^{j}},\tfrac{\partial\,}{\partial x^{k}}\big) -
\big({\nabla}^{\sigma}_{\frac{\partial\,}{\partial x^{j}}}{h}\big) \big(\tfrac{\partial\,}{\partial x^{i}},\tfrac{\partial\,}{\partial x^{k}}\big) \, = \, \big\langle \,{\Rm}
\big(\tfrac{\partial\,}{\partial x^{i}},\tfrac{\partial\,}{\partial x^{j}}\big) \tfrac{\partial\,}{\partial x^{k}}, \nu \, \big\rangle \,.
\end{equation}
Using the umbilicity property~\eqref{umbilic} of $L$ and tracing the
left hand side of equation~\eqref{codman} with the inverse of the metric $(g^\sigma)^{ik}_{\sigma} = g^{ik}$, we get
$$
g^{ik}\bigg[ \, \big({\nabla}^{\sigma}_{\frac{\partial\,}{\partial
    x^{i}}}{h}\big) \big(\tfrac{\partial\,}{\partial
  x^{j}},\tfrac{\partial\,}{\partial x^{k}}\big) -
\big({\nabla}^{\sigma}_{\frac{\partial\,}{\partial x^{j}}}{h}\big)
\big(\tfrac{\partial\,}{\partial x^{i}},\tfrac{\partial\,}{\partial
  x^{k}}\big) \,\bigg] = 
\, \frac{1}{n-1}\,\partial_{j}{\HHH} -
\partial_{j}{\HHH} \,=\, -\frac{n-2}{n-1} \, \partial_{j}{\HHH} \, .
$$
Tracing also the right hand side, we get
\begin{equation}\label{eq1000}
-\frac{n-2}{n-1} \, \partial_{j}{\HHH} \,=\, g^{ik} \, \big\langle \,{\Rm}
\big(\tfrac{\partial\,}{\partial x^{i}},\tfrac{\partial\,}{\partial x^{j}}\big) \tfrac{\partial\,}{\partial x^{k}}, \nu \, \big\rangle \, = \, \Ric \big(\tfrac{\partial}{\partial x^{j}},\nu\big)\, = 
\, \Ric_{0j} /\sqrt{g_{00}}\,,
\end{equation}
as $g^{i0}=0$ when $i\geq 1$ and $\big\langle \,{\Rm}
\big(\tfrac{\partial\,}{\partial x^{i}},\tfrac{\partial\,}{\partial
  x^{j}}\big) \tfrac{\partial\,}{\partial x^{k}}, \nu \, \big\rangle$
is equal to zero if $i=k=0$.\\
Now, it is a general fact (see Corollary~16.17 in~\cite{besse}) that
every Codazzi tensor $\TTT$ commutes with the Ricci tensor, that is,
$g^{kl}\TTT_{ik}\Ric_{lj}=g^{kl}\Ric_{ik}\TTT_{lj}$. In particular, 
$$
\rho \,  \Ric_{0j} \, = \, g^{kl}\TTT_{0k}\Ric_{lj} \, = \,
g^{kl}\Ric_{0k}\TTT_{lj} \, = \, \sigma \,
g^{kl}\Ric_{0k}g_{lj}=\sigma \, \Ric_{0j}\,,
$$
hence, $\Ric_{0j}=0$ for every $j = 1, \dots, n-1$, as $\rho
\not=\sigma$ in $U$. We conclude by equation~\eqref{eq1000} that the
mean curvature ${\HHH}$ is constant along every connected component of
$L$, hence, the same conclusion holds for any leaf of $V_{\sigma}$.

Next, we recall from Proposition 16.11 (ii) in~\cite{besse} that the
eigenvalue $\sigma$ is constant along the leaves of $V_\sigma$, thus,
in our local chart, it only depends on the $x^0$ variable. Moreover,
by the same proposition, one has that
\begin{equation}
\label{meancur}
{\HHH}\, = \,\frac{1 }{\rho -\sigma}\, \frac{\partial \sigma}{\partial x^0} \,.
\end{equation}
From this we deduce that the connected component of the
$V_\sigma$--leaves through critical points of $\sigma$ are minimal and
by the umbilicity they are also totally geodesic. This gives the
description at the point (3) of the (possibly non present) interior of the set
where $d\sigma=0$.

We pass now to consider the open set $W\subset M$ given by the
complement of the critical points of $\sigma$ in $M$. We are going to
prove that $\rho$ is locally constant on the connected component of
the $V_\sigma$--leaves which are sitting in $W$. To see this, it is
sufficient to take the (coordinate) derivative of both sides of relation~\eqref{meancur} 
with respect to $x^j$, for $j=1,\ldots,n-1$. This gives
$$
0=\partial_j{\HHH} \, = \,- \frac{1}{(\rho -\sigma)^2}\,  
\frac{\partial \rho}{\partial x^j}
\frac{\partial \sigma}{\partial x^0} + \frac{1}{\rho -\sigma}\, \frac{\partial^2 \sigma}{\partial x^j \partial x^0} \, = \,- \frac{1}{(\rho -\sigma)^2}\,  
\frac{\partial \rho}{\partial x^j}
\frac{\partial \sigma}{\partial x^0} \,,
$$
where we used the symmetry of the second derivative together with the
constancy of $\sigma$ along the $V_\sigma$--leaves.
Since in our coordinates $d\sigma = \partial_0 \sigma dx^0$ and 
$d\sigma \neq 0$ in $W$, the claim follows. To conclude, we
observe that the boundary of $W$ (if any) can be described as a
suitable union of connected component of level sets of $\sigma$. By
continuity the eigenvalue $\rho$ must be locally constant
also on $\partial W$.

To show that $g$ has a warped product structure on $\overline{W} = W
\cup \partial W$, we first observe that the condition $\partial_j
\rho= 0$, for $j=1, \ldots, n-1$, combined
with~\cite[Proposition 16.11--(ii)]{besse}, implies that $V_{\rho}$ is
a {\em geodesic} line distribution in $\overline{W}$. This means that
$\nabla_\nu \nu = 0$, which easily implies $\Gamma^{j}_{00}=0$, hence
$\partial_j g_{00}=0$, for every $j= 1,\ldots, n-1$. Equation~\eqref{umbilic} then yields
$$
\frac{\partial g_{ij}}{\partial x^0} \,=\, -2 \,{\Gamma}^{0}_{ij} =
\frac{2 ({\HHH}/{\sqrt{g_{00}}})}{n-1}\,  g_{ij}\,.
$$
Since ${\HHH}$ and $g_{00}$ are constant along $V_{\sigma}$, one has that
$$ 
\frac{\partial g_{ij}}{\partial x^0}(x^{0}, \dots, x^{n-1}) \,=\,
\varphi(x^{0}) \,g_{ij}(x^{0}, \dots, x^{n-1}) \,,
$$ 
for some function $\varphi$ depending only on the $x^{0}$
variable. Setting $\psi(x^0) = d \varphi / d x^0$, 
one has that $e^{-\psi} g_{ij}$ does not depend on the $x^0$
variable. Thus, for every $i,j = 1, \ldots, n-1$, we can write
$$
g_{ij}(x^{0}, \dots, x^{n-1})\,=\,e^{\psi(x^{0})} \, G_{ij}(x^{1}, \dots, x^{n-1}) \,,
$$
for some suitable functions $G_{ij}$. This prove that $g$ has a local
warped product structure in $\overline{W}$ and the proof is complete.
\end{proof}

\begin{rem}\label{remanal}
If the metric is analytic and the Riemannian manifold $(M^n,g)$ is connected, the presence of an open set where $\sigma$ is constant implies that everywhere $\sigma$ is constant and $d\sigma=0$, hence $W=\emptyset$. In the opposite case $\overline{W}=M^n$ and the totally geodesic hypersurfaces (where $\sigma$ is constant) whose union gives $\partial W$ are locally finite.\\
Hence, in the analytic case we have a dichotomy: either the whole manifold is locally a warped product or it is globally foliated by totally geodesic hypersurfaces.
\end{rem}

\section{An Example}
\label{example}

We show now that actually the two situations described in
Theorem~\ref{mainteo} can be both present in a Riemannian manifold if the metric is only smooth but not analytic.

We follow the line of Merton~\cite{merton}.

Let $M=\RR\times\SS^1\times\SS^1$ be endowed with the Riemannian metric
$$
g(t,x,y)=\bigl(\sigma(t)-\rho(t,x,y)\bigr)^{-2} dt^2+\sigma dx^2+\sigma dy^2\,,
$$
where $\sigma:\RR\to\RR^+$ and $\rho:M\to\RR$ are smooth functions,
such that:

\begin{itemize}
\item The function $\sigma$ is monotone increasing 
from 1 to 2, with $\sigma^\prime>0$, in the interval $(-\infty,-1)$, 
constant equal to 2 in the interval $[-1,1]$ and again monotone
increasing from 2 to 3, with $\sigma^\prime>0$, in the interval $(1,+\infty)$.

\item The function $\rho$ is equal to $3\sigma$ 
when $t\in(-\infty,-1]$ or $t\in[1,+\infty)$, for every $(x,y)\in\SS^1\times\SS^1$.

\item For $t\in(-1,1)$ and every $(x,y)\in\SS^1\times\SS^1$, the
  function $\rho$ is nonconstant
  on the leaves $\{t\}\times\SS^1\times\SS^1$, in particular it cannot
  be three times the function $\sigma$.
\end{itemize}

We then define the (1,1)--tensor $\TTT$ as follows
\begin{align*}
\TTT(\partial_t)=&\,\rho(t,x,y)\partial_t\\
\TTT(\partial_x)=&\,\sigma(t)\partial_x\\
\TTT(\partial_y)=&\,\sigma(t)\partial_y
\end{align*}
and we will show that $\TTT$ is a Codazzi tensor.\\
The (0,2)--version of $\TTT$ reads
\begin{align*}
\TTT_{tt}=&\,\TTT_t^t \, g_{tt}=\frac{\rho}{(\sigma-\rho)^2}\\
\TTT_{xx}=&\,\TTT_x^x \, g_{xx}=\sigma^2\\
\TTT_{yy}=&\,\TTT_y^y \, g_{yy}=\sigma^2\\
\end{align*}
and all the other components are null.

The Christoffel symbols of the metric $g$ are given by
\begin{align*}
\Gamma^t_{tt}=&\,-\bigl(\sigma-\rho\bigr)^{-1}(\sigma^\prime-\partial_t\rho)\\
\Gamma^i_{tt}=&\,-\sigma^{-1}\bigl(\sigma-\rho\bigr)^{-3}\partial_i\rho\\
\Gamma^t_{it}=&\,\bigl(\sigma-\rho\bigr)^{-1}\partial_i\rho\\
\Gamma^i_{jt}=&\,\sigma^{-1}\sigma^\prime\delta_j^i/2\\
\Gamma^t_{ij}=&\,-\bigl(\sigma-\rho\bigr)^{2}\sigma^\prime\delta_{ij}/2\\
\Gamma^k_{ij}=&\,0\,,
\end{align*}
where the indices $i,j,k$ can only be $x$ and $y$.\\
Thus we compute (we skip the trivial checks)
\begin{align*}
\nabla_y\TTT_{xx}-\nabla_x\TTT_{yx}
=&\,\partial_y\sigma^2-2\TTT_{xp}\Gamma^p_{xy}+\TTT_{yp}\Gamma^p_{xx}
+\TTT_{xp}\Gamma^p_{xy}\\
=&\,-\sigma^2\Gamma^x_{xy}+\sigma^2\Gamma^y_{xx}\\
=&\,0
\end{align*}

\begin{align*}
\nabla_t\TTT_{xx}-\nabla_x\TTT_{tx}
=&\,\partial_t\sigma^2-2\TTT_{xp}\Gamma^p_{xt}+\TTT_{tp}\Gamma^p_{xx}
+\TTT_{xp}\Gamma^p_{xt}\\
=&\,2\sigma\sigma^\prime-\sigma^2\Gamma^x_{xt}+\frac{\rho}{(\sigma-\rho)^2}\Gamma^t_{xx}\\
=&\,2\sigma\sigma^\prime-\sigma\sigma^\prime/2-\sigma^\prime\rho/2\\
=&\,\bigl(3\sigma-\rho\bigr)\sigma^\prime/2
\end{align*}
\begin{align*}
\nabla_x\TTT_{tt}-\nabla_t\TTT_{xt}
=&\,\partial_x\biggl(\frac{\rho}{(\sigma-\rho)^2}\biggr)
-2\TTT_{tp}\Gamma^p_{tx}+\TTT_{xp}\Gamma^p_{tt}
+\TTT_{tp}\Gamma^p_{xt}\\
=&\,\frac{\partial_x \rho}{(\sigma-\rho)^2}
+\frac{2\rho\,\partial_x\rho}{(\sigma-\rho)^3}
-\frac{\rho}{(\sigma-\rho)^2}\Gamma^t_{tx}+\sigma^2\Gamma^x_{tt}\\
=&\,\frac{\partial_x \rho}{(\sigma-\rho)^2}
+\frac{2\rho\,\partial_x\rho}{(\sigma-\rho)^3}
-\frac{\rho\,\partial_x\rho}{(\sigma-\rho)^3}
-\frac{\sigma\,\partial_x\rho}{(\sigma-\rho)^{3}}\\
=&\,0
\end{align*}
\begin{align*}
\nabla_t\TTT_{xy}-\nabla_x\TTT_{ty}
=&\,-\TTT_{xp}\Gamma^p_{ty}-\TTT_{yp}\Gamma^p_{tx}
+\TTT_{yp}\Gamma^p_{tx}+\TTT_{tp}\Gamma^p_{xy}\\
=&\,-\sigma^2\Gamma^x_{ty}+\frac{\rho}{(\sigma-\rho)^2}\Gamma^t_{xy}\\
=&\,0
\end{align*}
\begin{align*}
\nabla_x\TTT_{yt}-\nabla_y\TTT_{xt}
=&\,-\TTT_{yp}\Gamma^p_{tx}-\TTT_{tp}\Gamma^p_{xy}
+\TTT_{xp}\Gamma^p_{ty}+\TTT_{tp}\Gamma^p_{xy}\\
=&\,-\sigma^2\Gamma^y_{tx}+\sigma^2\Gamma^x_{ty}\\
=&\,0\,.
\end{align*}
Hence, by our choices for the functions $\sigma$ and $\rho$, the 
tensor $\TTT$ is a Codazzi tensor.

It is easy to see that in the zone where $\sigma$ is nonconstant, the
manifold is a warped product on a interval of $\RR$, instead, in the zone
$(-1,1)\times\SS^1\times\SS^1$, if the function $\rho$ is suitably
chosen nonconstant on the leaves  $\{t\}\times\SS^1\times\SS^1$, it can be
checked that $(M,g)$ is not a warped product on an interval (actually,
in this example, it is incidentally a warped product on  $\SS^1\times\SS^1$), see
the careful analysis in~\cite{merton}.\\
Hence, the two situations described in Theorem~\ref{mainteo} are both
present in this example.

\section{Three--Dimensional Gradient Ricci Solitons}

Let $(M^{3},g)$ be a three--dimensional gradient Ricci soliton, that is a Riemannian manifold satisfying the equation
\begin{equation}\label{sol}
\Ric + \nabla^{2} f \,=\, \lambda \,g
\end{equation}
for some smooth function $f:M^3\to\RR$ and some constant $\lambda\in\RR$.

\begin{lemma} On every three--dimensional gradient Ricci soliton the tensor
$$
\TTT \,=\, \big( \Ric - \tfrac{1}{2}\RRR \, g \big) e^{-f}
$$
is a Codazzi tensor.
\end{lemma}

\begin{proof} Let $(M^{3},g)$ be a three dimensional gradient Ricci soliton satisfying equation~\eqref{sol} and let 
$$
\TTT_{ij} \,=\, \big( \RRR_{ij}- \tfrac{1}{2}\RRR\,g_{ij}\big)\,e^{-f} \,.
$$
We want to prove that $\TTT$ is a Codazzi tensor, i.e. we have to show that
$$
\nabla_{k}\TTT_{ij} \,=\, \nabla_{j} \TTT_{ik} \,,
$$
for every $i,j,k=1,2,3$. One has
\begin{eqnarray}\label{ids1}
\nabla_{k}\TTT_{ij} - \nabla_{j} \TTT_{ik} &=& \big[\nabla_{k}\RRR_{ij}-\nabla_{j}\RRR_{ik} -\tfrac{1}{2}(\nabla_{k}\RRR\,g_{ij}-\nabla_{j}\RRR\,g_{ik})\big]\,e^{-f}\nonumber\\
&& + \,\big[\tfrac{1}{2}\RRR(\nabla_{k}f\,g_{ij}-\nabla_{j}f\,g_{ik})-\nabla_{k}f\,\RRR_{ij}+\nabla_{j}f\,\RRR_{ik} \big]\,e^{-f}\,.
\end{eqnarray}
On the other hand, the following two identities hold on any gradient Ricci soliton (for a proof, see~\cite{mantemin2}, for instance)
\begin{eqnarray}\label{eqs1}
& \nabla_{k} \RRR \,=\, 2 \nabla_{p}f\, \RRR_{pk}& \\ \label{eqs2}
& \nabla_{k}\RRR_{ij} - \nabla_{j}\RRR_{ik} \,=\, -\RRR_{kjip}\nabla_{p}f \,. &
\end{eqnarray}
Moreover, since we are in dimension three, one has the decomposition of the Riemann tensor
$$
\RRR_{kjip} \,=\, \RRR_{ik} g_{jp}-\RRR_{kp}g_{ij}+\RRR_{jp}g_{ik}-\RRR_{ij}g_{kp}- \tfrac{1}{2}\RRR(g_{ik}g_{jp}-g_{ij}g_{kp}) \,.
$$
Combining with equation~\eqref{eqs2}, we obtain
$$
\nabla_{k}\RRR_{ij} - \nabla_{j}\RRR_{ik} \,=\, -\nabla_{j}f\,\RRR_{ik} +\nabla_{p}f\,\RRR_{kp}g_{ij}-\nabla_{p}f\,\RRR_{jp}g_{ik}+\nabla_{k}f\, \RRR_{ij}-\tfrac{1}{2}\RRR(\nabla_{k}f\,g_{ij}-\nabla_{j}f\,g_{ik}) \,. 
$$
Hence, substituting this in equation~\eqref{ids1} and using relation~\eqref{eqs1}, we immediately get
\begin{eqnarray*}
\nabla_{k}\TTT_{ij} - \nabla_{j} \TTT_{ik} \,=\,0\,.
\end{eqnarray*}
\end{proof}

As an application of this lemma and the results of the previous
sections, we have the following theorem.

\begin{teo}\label{teosol} Let $(M^3,g)$ be a complete,
  three--dimensional, simply connected Riemannian manifold, which is a steady, gradient Ricci soliton and 
  assume that there exists an open subset $U\subset
  M$, where the Ricci tensor of $g$ has at most two distinct
  eigenvalues. Then, either the manifold splits a line or it is
  locally conformally flat.
\end{teo}
\begin{proof} 
By~\cite{zhang2}, as $(M^3,g)$ is complete, this 
gradient Ricci soliton generates an ancient Ricci flow. Then, by the
result~\cite[Corollary~2.4]{chen2} the evolving manifold, hence the Ricci
soliton, must have nonnegative sectional curvatures. Moreover, it is
well known, by the properties of the parabolic equations, that the
metric $g$ must be analytic (see~\cite[Chapter~3, Section~2]{chowbookII}).

If at least one sectional curvature is zero at some point, then the manifold $(M^3,g)$ "splits a line" (see~\cite{chowbookII}), that is, it is isometric to the Riemannian product of $\RR$ with a surface. Hence, we will assume in the rest
  of the proof that all the sectional  curvatures are strictly
  positive everywhere.
  
The analyticity of the metric implies that either at every point of the open subset $U$ all of the three eigenvalues of the Ricci tensor coincide, or there is another, possibly smaller, open subset $W$ of $M^3$ such that the Ricci tensor has everywhere in $W$ exactly two distinct eigenvalues. 

The first case cannot occur, since$(U,g)$ would be locally isometric to an Einstein manifold with positive curvature, then, by
analyticity, $(M^{3},g)$ must be isometric to the sphere
$\SS^3$, and this would contradict the fact that every compact steady Ricci soliton is Ricci flat. Thus, we assume from now on that there exists an open subset $W$ where the Ricci tensor has exactly two distinct eigenvalues. This implies that on $W$ the Codazzi tensor $\TTT=(\Ric-\RRR g/2)e^{-f}$ has
two distinct eigenvalues $\sigma$, with multiplicity $2$, and $\rho$, with multiplicity $1$.

As Remark~\ref{remanal} applies to this case, by
Theorem~\ref{mainteo}, we have that two possible subcases:
either around every point of $W$ the manifold is locally isometric to a warped product of a surface on an interval, or the eigenvalue $\sigma$ of the Codazzi tensor $\TTT$ is constant on $W$,
hence on the whole $M^3$ by analyticity. 
On the other hand, since the curvature of $(M^{3},g)$ is strictly
positive, $(M^{3},g)$ cannot admit equidistant totally geodesic
submanifolds of dimension greater than one by the second Rauch
comparison theorem (see~\cite{groklinmey}) 
and this latter subcase is excluded.

In the former subcase, the leaves of the distribution $V_\sigma$ are
umbilical and the two eigenvalues of $\TTT$ are constant on every
leaf. Let $L$ be a connected component of a leaf of $V_{\sigma}$. By Gauss formula, one
has that the scalar curvature of the induced metric $g^{\sigma}$ is
given by
\begin{equation}
\RRR^{\sigma} \,=\, \RRR - 2\RRR_0^0+\HHH^{2}/2\,,\label{rsigmacon}
\end{equation} 
where $\HHH$ denotes the mean curvature of $L$. Since $g$ has positive
sectional curvature, one has $\RRR g-2\,\Ric>0$ and we obtain that
$\RRR^{\sigma}$ is positive. We want to prove that $g$ is locally a
warped product on an interval of $\RR$ of two--dimensional fibers with 
constant positive curvature. Hence, we have to show that $\RRR^{\sigma}$ is
constant on $L$. 

First of all, we observe that by the same reasoning as
in the proof of Theorem~\ref{mainteo}, we know that $\HHH$ is constant on
$L$. Thus, it remains to show that also the quantity $\RRR-2\RRR_0^0$ is
constant on $L$. The fact that all the eigenvalues of the tensor $\TTT$ are constant on $L$,
implies that the trace of $\TTT$
$$
\trace(\TTT) \,=\, -\RRR e^{-f}/2
$$ 
is constant on $L$. We claim that also $f$ has to be constant on
$L$. 

Using the adapted coordinate system as in the proof of
Theorem~\ref{mainteo}, we assume by contradiction that $\partial_{j}
f\neq 0$, for some $j\in\{1,2\}$ at some point of $L$. 
As $\partial_{j}\trace(\TTT)=0$ it follows 
that $\nabla f$ and $\nabla\RRR$ are parallel and
$$
\RRR\nabla f\,=\, \nabla\RRR \,=\, 2\Ric(\nabla f , \cdot)\,, 
$$
hence, $\nabla f$ is an eigenvalue of the Ricci tensor. Being $\partial_j
f\not=0$, then it must be
$$
\RRR\nabla f\,=\, 2\Ric_j^j\nabla f\,, 
$$
which is a contradiction, as $\RRR=2\Ric_j^j+\Ric_0^0$ and $\Ric_0^0$
is positive by assumption (notice that we have used the fact that the
Ricci tensor has exactly two distinct eigenvalues with the same
eigendistributions as the Codazzi tensor $\TTT$). Thus, we have proved
that $f$ is constant on $L$ which implies that $\RRR$ is constant on $L$ too. Then, it follows, by the definition of $\TTT$ and the fact that its
eigenvalues are constant on $L$, that also the 
eigenvalues of the Ricci tensor are constant on $L$. In particular, $\RRR_0^0$ is constant on $L$. 

By relation~\eqref{rsigmacon}, we conclude that $L$ has positive constant scalar curvature $\RRR^{\sigma}$. Hence, the leaf $L$ is locally isometric to $\SS^{2}$ and the metric $g$ in $W$ is locally a warped product of an interval with two--dimensional spherical fibers. In particular it is locally conformally flat.

Using once again the analyticity, we can conclude that 
since $(W, g|_W)$ is a locally conformally flat open subset of $(M^3,g)$, then the whole $(M^3,g)$ must be locally conformally flat. This completes the proof.
\end{proof}

\begin{rem} By the same argument, 
the conclusion of this theorem also holds for 
complete, three--dimensional, simply connected, {\em expanding}, 
gradient Ricci solitons with nonnegative sectional curvatures.
\end{rem}

\begin{rem} Three--dimensional, locally conformally flat, gradient steady Ricci solitons were classified by Cao--Chen~\cite{caochen}. In particular,
  under the assumptions of Theorem~\ref{teosol} we have that
  $(M^{3},g)$ is isometric to $\RR^{3}$, the {\em Bryant} soliton or the Riemannian product of $\RR$ with {\em Hamilton's cigar}. Note that if a three--dimensional, gradient steady Ricci
  soliton splits a line, then it must be the Riemannian product of 
  $\RR$ with a two--dimensional complete simply connected gradient steady soliton, that is, $\RR^2$ or Hamilton's cigar.
\end {rem}

\

\bibliographystyle{amsplain}
\bibliography{biblio}

\

\end{document}